\documentclass[10pt]{amsart}
\usepackage{amssymb}
\usepackage{bm}
\usepackage{graphicx}
\usepackage[centertags]{amsmath}
\usepackage{amsfonts}
\usepackage{amsthm}
\usepackage{amsbsy}
\usepackage{mathtools}
\usepackage{mathrsfs}
\usepackage{xfrac}
\usepackage{color}
\usepackage{url}
\usepackage{amsbsy}
\usepackage{cases}
\usepackage[all]{xy}
\usepackage[colorlinks,linkcolor=black,urlcolor=black,citecolor = blue]{hyperref}
\usepackage{comment}
\newtheorem{thm}{Theorem}[section]

\newtheorem{lem}[thm]{Lemma}

\newtheorem{prop}[thm]{Proposition}

\newtheorem{fact}[thm]{Fact}

\newtheorem{defn}[thm]{Definition}
\theoremstyle{definition}

\newcommand{\rr}{\mathbb{R}}
\newcommand{\nn}{\mathbb{N}}
\newcommand{\ee}{\varepsilon}

\newcommand{\meg}{\geqslant}
\newcommand{\mik}{\leqslant}
\newcommand{\ave}{\mathbb{E}}

\newcommand{\cala}{\mathcal{A}}

\newcommand{\calf}{\mathcal{F}}

\newcommand{\cals}{\mathcal{S}}
\newcommand{\calp}{\mathcal{P}}
\newcommand{\calq}{\mathcal{Q}}

\newcommand{\bmu}{\boldsymbol{\mu}}

\begin{document}

\title{An algorithmic regularity lemma for $L_p$ regular sparse matrices}

\author{Thodoris Karageorgos and Silouanos Brazitikos}

\address{Department of Mathematics, University of Athens, Panepistimiopolis 157 84, Athens, Greece}
\email{silouanb@math.uoa.gr}
\address{Department of Mathematics, University of Athens, Panepistimiopolis 157 84, Athens, Greece}
\email{tkarageo@math.uoa.gr}

\thanks{2010 \textit{Mathematics Subject Classification}: 05C35, 46B25, 60G42, 68R05.}
\thanks{\textit{Key words}: sparse graphs, sparse matrices, regularity lemma, algorithmic regularity lemma.}


\begin{abstract}
We prove an algorithmic regularity lemma for $L_p$ regular matrices $(1 < p \mik \infty),$ a class of sparse $\{0,1\}$ matrices
which obey a natural pseudorandomness condition. This extends a result of Coja-Oghlan, Cooper and Frieze who treated the case
of $L_{\infty}$ regular matrices. We also present applications of this result for tensors and  MAX-CSP instances.
\end{abstract}

\maketitle


\section{Introduction}

\numberwithin{equation}{section}

\subsection{Overview}

It is well known that it is NP-hard not only to compute the optimal solution for the  MAX-CSP problem,
but also to find ``good'' approximations of this optimal solution (see, e.g., \cite{Ha,KKMO,TSSW}).

In a seminal paper  \cite{FK}, Frieze and  Kannan proved several results concerning  \emph{dense} instances of the previous problems.
Later on, Coja-Oghlan,  Cooper and  Frieze \cite{CCF} showed that such results may be extended to the \emph{sparse} setting if we assume
a pseudorandomness condition known as \emph{$(C,\eta)$-boundedness} (see \cite{Koh,Koh1}). Specifically, in \cite{CCF} the authors found
an algorithm for approximating a sparse $\{0,1\}$ matrix $f$ by a sum of \emph{cut matrices} under the assumption that $f$ is $(C,\eta)$-bounded.
The crucial fact is that the number of summands is independent of the size of the matrix and its density. Then, using this result, they proved
a similar theorem for tensors which in turn yields approximations for sparse MAX-CSP instances.

The purpose of this paper is to extend these results to a larger class of  sparse $\{0,1\}$ matrices, namely, the $L_p$ regular matrices
introduced recently by Borgs, Chayes, Cohn and Zhao \cite{BCCZ}.

\subsubsection{ \ }

To proceed with our discussion it is useful at this point to introduce some pieces of notation and some terminology.
Unless otherwise stated, in the rest of this paper by $n_1$ and $n_2$ we denote two positive integers. As usual,
for every positive integer $n$ we set $[n]\coloneqq \{1,\dots,n\}$. The cardinality of a finite set $S$ is denoted by~$|S|$.

If $X$ is a nonempty finite set, then by $\mu_X$ we denote the uniform probability measure on $X$, that is, $\mu_X(A)\coloneqq |A|/|X|$
for every $A\subseteq X$. For notational simplicity, the probability measures $\mu_{[n_1]}, \mu_{[n_2]}$ and $\mu_{[n_1]\times [n_2]}$
will be denoted by $\mu_1, \mu_2$ and $\bmu$  respectively. If $\calp$ is a partition of $[n_1]\times [n_2],$ then by $\cala_\calp$
we denote the (finite) $\sigma$-algebra on $[n_1]\times [n_2]$ generated by $\calp.$

Next, let $X_1,X_2$ be nonempty finite sets and set
\[\cals_{X_1 \times X_2}\coloneqq \{A_1 \times A_2 \colon A_1 \subseteq X_1 \ \text{and} \ A_2 \subseteq X_2\}.\]
If $X_1$ and $X_2$ are understood from the context (in particular, if $X_1=[n_1]$ and $X_2=[n_2]$), then we shall denote $\cals_{X_1 \times X_2}$
simply by $\cals.$ Moreover, for every partition $\calp$ of $X_1 \times X_2$ with $\calp \subseteq \cals_{X_1 \times X_2}$ we set
\[\iota(\calp)\coloneqq \min\big\{ \min\{\mu_{X_1}(P_1), \mu_{X_2}(P_2)\}: P=P_1 \times P_2 \in \calp\big\}.\]
Namely, the quantity $\iota(\calp)$ is the minimal density of each side of each rectangle $P_1\times P_2$ belonging to the partition $\calp$.

Now recall that a \emph{cut matrix} is a matrix $g \colon [n_1] \times [n_2] \to \rr$ for which there exist two sets $S \subseteq [n_1]$ and
$T \subseteq [n_2]$, and a real number $c$ such that $g=c \cdot \boldsymbol{1}_{S \times T};$ the set $S \times T$ is called the \emph{support}
of the matrix $g$. Also recall that for every matrix $f \colon [n_1] \times [n_2] \to \rr$ the \emph{cut norm} of $f$ is the quantity
\[\|f\|_{\square} = \max_{\substack{S \subseteq [n_1]\\ T \subseteq[n_2]}}\, \Big|\sum_{(x_1,x_2) \in S \times T} f(x_1,x_2) \Big|=
(n_1\cdot n_2)  \cdot  \max_{\substack{S \subseteq [n_1]\\ T \subseteq[n_2]}} \, \Big| \int_{S \times T} \! f \, d\bmu \Big| .\]

Finally, let  $f \colon [n_1] \times [n_2] \to \{0,1\}$ be a matrix and let $\calp$ be a partition of $[n_1] \times [n_2]$ with $\calp \subseteq \cals.$
Recall that the \emph{conditional expectation} of $f$ with respect to $\cala_\calp$ is defined by
\[\ave(f \, | \,\cala_\calp) = \sum_{P \in \calp} \frac{\int_P f \,d\bmu}{\bmu(P)}\, \boldsymbol{1}_{P}.\]
Notice, in particular, that $\ave(f \, |\,\cala_\calp)$ is a sum of cut matrices with disjoint supports; this observation will be useful later on.
Also note that if $1 \mik p < \infty,$  then we have
\[\|\ave(f \, | \,\cala_\calp)\|_{L_p} =\Big(\sum_{P \in \calp}\Big|\frac{\int_P f \,d\bmu}{\bmu(P)}\Big|^{p}\, \bmu(P)\Big)^{1/p}\] while if
$p = \infty,$ then
\[\|\ave(f \, | \,\cala_\calp)\|_{L_\infty} = \max\Big\{ \Big|\frac{\int_P f \,d\bmu}{\bmu(P)}\Big|: P \in \calp\Big\}. \]
In particular, observe that $\|f\|_{L_1}$ is equal to the \emph{density} of $f$, that is, the number of ones in the matrix divided by $n_1\cdot n_2$.
Also notice that $\|f\|_{\square} = \|f\|^p_{L_p} \cdot (n_1\cdot n_2)$ for every $1\mik p<\infty$.

\subsubsection{ \ }

We are now in a position to introduce the class of $\{0,1\}$ matrices which we consider in this paper.
\begin{defn}[$L_p$ regular matrices \cite{BCCZ}] \label{defn1.1}
Let $0 < \eta \mik 1,$ $C\meg 1$ and $1 \mik p \mik \infty$. A~matrix $f \colon [n_1] \times [n_2] \to \{0,1\}$ is called \emph{$(C,\eta,p)$-regular}
$($or simply \emph{$L_p$ regular} if $C$ and $\eta$ are understood$)$ if for every partition $\calp$ of $[n_1] \times [n_2]$ with
$\calp \subseteq  \cals$ and $\iota(\calp) \meg \eta$ we have
\begin{equation} \label{eq1.1}
\|\ave(f \,|\,\cala_\calp) \|_{L_p} \mik C\, \|f\|_{L_1}.
\end{equation}
\end{defn}
Notice that, by the monotonicity of the $L_p$ norms, if $1 \mik p_1 \mik p_2 \mik \infty$ and $f$~is~$L_{p_2}$~regular, then $f$ is $L_{p_1}$ regular.
Thus, $L_p$ regularity is less restrictive when $p$ gets smaller. Also observe that for $p=1$ the previous definition is essentially of no interest
since  every $\{0,1\}$ matrix is $L_1$ regular. On the other hand, the case $p=\infty$ in Definition \ref{defn1.1} is equivalent to the
aforementioned $(C,\eta)\text{-boundedness}$ condition. Indeed, recall that a matrix $f \colon [n_1]\times [n_2] \to \{0,1\}$ is said to be
$(C,\eta)$-bounded if for every $S \subseteq [n_1] $ and every $T \subseteq [n_2]$ with $\mu_1(S) \meg \eta$ and $\mu_2(T) \meg \eta$ we have
\[ \frac{\int_{S \times T} f \,d\bmu} {\bmu(S \times T)} \mik C \, \|f\|_{L_1}. \]
We have the following simple fact. (See also Lemma \ref{lem3.1} below.)
\begin{fact}
Let\, $0 < \eta \mik 1$ and $C\meg 1$, and let $f \colon [n_1]\times [n_2] \to \{0,1\}$ be a matrix. If $f$ is $(C,\eta)$-bounded,
then $f$ is $(C,\eta,\infty)$-regular. Conversely, if $f$ is $(C,\eta,\infty)$-regular, then $f$ is $(4C,\eta)$-bounded.
\end{fact}
Between the extreme cases ``$p=1$" and ``$p=\infty$"\!, there is a large class of sparse matrices which are very well behaved.
The examples which are easiest to grasp are random. Specifically, by \cite[Theorem 2.14]{BCCZ}, for every symmetric measurable
function $W\colon [0,1]\times [0,1]\to \rr^+$ with $W\in L_p$ $(1< p\mik \infty)$ and every positive integer $n$ there exists
a natural model\footnote{This model encompasses the classical Erd\H{o}s--R\'{e}nyi model---see, e.g., \cite{BJR}.} of sparse random
$n$-by-$n$ $\{0,1\}$ matrices which are $L_p$~regular asymptotically almost surely. (On the other hand, if $W\notin L_p$, then a typical
matrix in this model in not $L_p$ regular.) Further (deterministic) examples, which are relevant from a number theoretic perspective,
are given in \cite{DKK3}.

\subsection{The main result}

The following theorem is the main result of this paper.
\begin{thm} \label{thm1.3}
There exist absolute constants $a_1,a_2 > 0$, an algorithm and a polynomial\,\footnote{Here, and in the rest of this paper,
by the term \textit{polynomial} we mean a real polynomial $\Pi$ with non-negative coefficients, that is, $\Pi(x)=a_dx^d+\dots+a_1 x+a_0$
where $d\in\nn$ and $a_0,\dots,a_d\in\rr^+$. Moreover, unless otherwise stated, we will assume that the degree $d$ and
the coefficients $a_0,\dots,a_d$ are absolute and independent of the rest of the parameters.} $\Pi_0$ such that the following holds.
Let $0 < \ee < 1/2$ and $C\meg 1$. Also let $1 < p \mik \infty$, set $p^\dagger=\min \{2,p\}$ and let $q$ denote
the conjugate exponent of $p^\dagger$  $($that is, $1/p^\dagger + 1/q =1$$)$. We set
\begin{equation} \label{eq1.2}
\tau= \Big\lceil\frac{a_1 \cdot C^2}{(p^\dagger-1) \, \ee^2}\Big\rceil \ \text{ and } \
\eta= \left(\frac{a_2 \cdot \ee}{C}\right)^{\sum_{i=1}^{\tau+1}(\frac{2}{p^\dagger}+1)^{i-1}q^i}.
\end{equation}
If\, we input
\begin{itemize}
\item[$\mathtt{INP}$:] a $(C,\eta,p)$-regular matrix $f \colon [n_1] \times [n_2] \to \{0,1\}$,
\end{itemize}
then the algorithm outputs
\begin{itemize}
\item[$\mathtt{OUT}$:] a partition $\calp $ of $[n_1] \times [n_2]$ with $\calp \subseteq \cals$, $|\calp| \mik 4^\tau$
and $\iota(\calp) \meg \eta,$ such that
\begin{equation} \label{eq1.3}
\|f - \ave(f \,|\,\cala_\calp)\|_{\square} \mik \ee \|f\|_{\square}.
\end{equation}
\end{itemize}
This algorithm has running time $(\tau\, 4^\tau)\cdot \Pi_0(n_1 \cdot n_2).$
\end{thm}
Theorem \ref{thm1.3} extends \cite[Theorem 1]{CCF} which corresponds to the case $p=\infty$\footnote{Actually, the argument in \cite{CCF}
works for the more general case $p\meg 2$. We also remark that the cut matrices obtained by \cite[Theorem 1]{CCF} do not necessarily have disjoint supports,
but this can be easily arranged---see \cite[Corollary 1]{CCF} for more details.}. Note that, by \eqref{eq1.2} and \eqref{eq1.3}, the matrix $f$ is well
approximated by a sum of at most $4^\tau$ cut matrices with disjoint supports and, moreover, the positive integer $\tau$ is independent of the size
of $f$ and its density. Also observe that, as expected, the running time of the algorithm in Theorem \ref{thm1.3} increases as $p$ decreases to $1$.

\subsection{Organization of the paper}

The paper is organized as follows. In Section 2 we recall some results which are needed for the proof of Theorem \ref{thm1.3}, and in Section 3
we present some preparatory lemmas. The proof of Theorem \ref{thm1.3} is completed in Section 4. Finally, in Section 5 we present applications
for tensors and sparse MAX-CSP instances.


\section{Background material}

\numberwithin{equation}{section}

\subsection{Martingale difference sequences}

Recall that a finite sequence $(d_i)_{i=0}^n$ of integrable real-valued random variables on  a probability space $(X,\Sigma,\mu)$ is said to be
a \emph{martingale difference sequence} if there exists a martingale $(f_i)_{i=0}^n$ such that $d_0=f_0$ and $d_i=f_i-f_{i-1}$ if $n\meg 1$ and
$i\in [n]$. We will need the following result due to Ricard and Xu \cite{RX} which can be seen as an extension of the basic fact that martingale
difference sequences are orthogonal in $L_2$. (See also \cite[Appendix A]{DKK1} for a discussion on this result and its proof.)
\begin{prop} \label{prop2.1}
Let $(X,\Sigma,\mu)$ be a probability space and $1<p \mik 2$. Then for every martingale difference
sequence $(d_i)_{i=0}^n$ in $L_p(X,\Sigma,\mu)$ we have
\begin{equation} \label{eq2.1}
\Big( \sum_{i=0}^n \|d_i\|^2_{L_p} \Big)^{1/2} \mik \Big(\frac{1}{p-1}\Big)^{1/2} \, \big\| \sum_{i=0}^n d_i\big\|_{L_p}.
\end{equation}
\end{prop}
We point out that the constant $(p-1)^{-1/2}$ appearing in the right-hand side of~\eqref{eq2.1} is best possible.

\subsection{The algorithmic version of Grothendieck's inequality}

We will need the following result due to Alon and Naor \cite{AN}.
\begin{prop} \label{Naor}
There exist a constant $a_0 >0,$ an algorithm and a polynomial $\Pi_{\mathrm{AN}}$ such that the following holds.
If we input
\begin{itemize}
\item[$\mathtt{INP}$:] a matrix $f \colon [n_1] \times [n_2] \to \rr$,
\end{itemize}
then the algorithm outputs
\begin{itemize}
\item[$\mathtt{OUT}$:] a set $A \in \cals$ such that $(n_1\cdot n_2) \big|\int_A f\, d\bmu \big|\meg  a_0 \|f\|_{\square}$.
\end{itemize}
This algorithm has running time $\Pi_{\mathrm{AN}}(n_1\cdot n_2)$.
\end{prop}
The constant $a_0$ in Proposition \ref{Naor} is closely related to Grothendieck's constant~$K_G$ (see, e.g., \cite{PG}).


\section{Preparatory Lemmas}

\numberwithin{equation}{section}

In this section we prove some preparatory results concerning $L_p$ regular matrices. We begin with the following lemma.
\begin{lem} \label{lem3.1}
There exist an algorithm and a polynomial\, $\Pi_1$ such that the following holds. Let $X_1,X_2$ be nonempty finite sets,
and let\, $0 < \vartheta <1/2$. If we input
\begin{itemize}
\item[$\mathtt{INP}$:] two sets $A_1 \subseteq X_1$ and $A_2 \subseteq X_2$ with $\mu_{X_1}(A_1 ) \meg \vartheta$ and
$\mu_{X_2}(A_2) \meg \vartheta$,
\end{itemize}
then the algorithm outputs
\begin{itemize}
\item[$\mathtt{OUT1}$:] a partition $\calq \subseteq \cals$ with $|\calq| \mik 4$ and $\iota(\calq) \meg \vartheta,$ and
\item[$\mathtt{OUT2}$:] a set $B\in \calq$ such that $A_1 \times A_2 \subseteq B$ and
$\mu_{X_1\times X_2}\big(B \setminus (A_1 \times A_2)\big) \mik 2\vartheta.$
\end{itemize}
This algorithm has running time\, $\Pi_1(|X_1| \cdot |X_2|).$
\end{lem}
\begin{proof}
We distinguish the following four (mutually exclusive) cases.
\medskip

\noindent \textsc{Case 1:} \textit{We have $\mu_{X_1}(A_1) < 1- \vartheta$ and $\mu_{X_2}(A_2) < 1- \vartheta$.} In this case the algorithm outputs
$\calq=\{A_1 \times A_2, (X_1 \setminus A_1) \times A_2, A_1 \times (X_2 \setminus A_2), (X_1 \setminus A_1) \times (X_2 \setminus A_2)\}$
and $B=A_1 \times A_2$. Notice that $\calq$ and $B$ satisfy the requirements of the lemma.
\medskip

\noindent \textsc{Case 2:} \textit{We have $\mu_{X_1}(A_1) < 1- \vartheta$ and $\mu_{X_2}(A_2) \meg 1- \vartheta$.} In this case the algorithm
outputs $\calq=\{A_1 \times X_2, (X_1 \setminus A_1 ) \times X_2\}$ and $B=A_1 \times X_2$. Again, it is easy to see that $\calq$ and $B$ satisfy
the requirements of the lemma.
\medskip

\noindent \textsc{Case 3:} \textit{We have $\mu_{X_1}(A_1) \meg 1- \vartheta$ and $\mu_{X_2}(A_2) < 1- \vartheta$.} This case is similar
to Case~2. In particular, we set $\calq=\{X_1 \times A_2, X_1\times (X_2 \setminus A_2 )\}$ and $B=X_1 \times A_2$.
\medskip

\noindent \textsc{Case 4:} \textit{We have $\mu_{X_1}(A_1) \meg 1- \vartheta$ and $\mu_{X_2}(A_2) \meg 1- \vartheta$.} In this case the algorithm
outputs $\calq=\{X_1 \times X_2\}$ and $B=X_1 \times X_2$. As before, it is easy to see that $\calq$ and $B$ are as desired.
\medskip

Finally, notice that the most costly part of this algorithm is to estimate the quantities $\mu_{X_1}(A_1)$ and $\mu_{X_2}(A_2)$, but of course
this can be done in polynomial time of $|X_1|\cdot |X_2|$. Thus, this algorithm will stop in polynomial time of $|X_1|\cdot |X_2|$.
\end{proof}
The next result is a H\"{o}lder-type inequality for $L_p$ regular matrices. To motivate this inequality, let $f\colon [n_1]\times [n_2]\to \{0,1\}$
be a matrix, let $1<p<\infty$, let $q$ denote its conjugate exponent and observe that, by H\"{o}lder's inequality, for every
$A\subseteq [n_1]\times [n_2]$ we have
\begin{equation} \label{e3.1}
\int_A f\, d\bmu \mik \|f\|_{L_1}^{1/p} \cdot \bmu(A)^{1/q}.
\end{equation}
Unfortunately, this estimate is not particularly useful if $f$ is sparse---that is, in the regime $\|f\|_{L_1}=o(1)$---since in this case
the quantity $\|f\|_{L_1}^{1/p}$ is \textit{not} comparable to the density $\|f\|_{L_1}$ of $f$. Nevertheless, we can improve upon \eqref{e3.1}
provided that the matrix $f$ is $L_p$ regular and $A\in\mathcal{S}$. Specifically, we have the following lemma (see also \cite[Proposition 4.1]{DKK2}).
\begin{lem} \label{lem3.2}
Let $0 < \eta <1/2$ and  $C\meg 1$. Also let $1< p \mik 2$ and let $q$ denote its conjugate exponent. Finally, let $f \colon [n_1]\times [n_2] \to \{0,1\}$
be  $(C,\eta,p)$-regular. Then for every $A \subseteq [n_1] \times [n_2]$ with $A \in \cals$ we have
\begin{equation} \label{eq3.2}
\int_A f \, d\bmu \mik C\, \|f\|_{L_1} (\bmu(A) + 6\eta)^{1/q}.
\end{equation}
\end{lem}
\begin{proof}
Fix a nonempty subset $A$ of $[n_1] \times [n_2]$ with $A \in \cals$, and let $A_1\subseteq [n_1]$ and $A_2 \subseteq [n_2]$ such that
$A =A_1 \times A_2.$ If $\mu_1(A_1) \meg \eta$ and $\mu_2(A_2) \meg \eta,$ then we claim~that
\begin{equation} \label{eq3.3}
\int_A f \,d\bmu \mik C\, \|f\|_{L_1} ( \bmu(A) +2\eta)^{1/q}.
\end{equation}
Indeed, by Lemma \ref{lem3.1} applied for $X_1=[n_1]$ and $X_2=[n_2],$ we obtain a partition $\calq$ of $[n_1] \times [n_2] $ with $\calq \in \cals$
and $\iota(\calq)\meg \eta$, and a set $B \in \calq$ such that $A \subseteq B$ and $\bmu(B \setminus A) \mik 2\eta.$ By the $L_p$ regularity of $f$, we have
\[\frac{\int_B  f \,d\bmu}{\bmu(B)} \, \bmu(B)^{1/p} \mik \|\ave(f \, |\, \cala_\calq)\|_{L_p} \mik C\, \|f\|_{L_1} \]
and so
\[ \int_A f \, d\bmu \mik \int_B f \,d\bmu \mik C\, \|f\|_{L_1} \bmu(B)^{1/q} \mik C\, \|f\|_{L_1} (\bmu(A) + 2\eta)^{1/q}. \]

Next, we assume that $\mu_1 (A_1) \meg \eta$ and $\mu_2(A_2) < \eta$ and observe that we may select a set $B \subseteq [n_2]$ with $\eta <\mu_2(B) \mik 2\eta.$
Then, we have
\begin{eqnarray*}
\int_A f\, d\bmu & \mik & \int_{A_1 \times (A_2 \cup B)}  f \,d\bmu \stackrel{\eqref{eq3.3}}{\mik}
C\, \|f\|_{L_1} \big(\bmu\big(A_1 \times (A_2 \cup B)\big) + 2\eta\big)^{1/q}\\
& \mik & C\, \|f\|_{L_1} (\bmu(A) + 2\eta\, \mu_1(A_1) + 2\eta)^{1/q}  \mik  C\, \|f\|_{L_1} (\bmu(A) + 4\eta)^{1/q}.
\end{eqnarray*}
The case $\mu_1(A_1) < \eta$ and $\mu_2(A_2) \meg \eta$ is identical.

Finally, assume that $\mu_1(A_1) < \eta$ and $\mu_2(A_2) < \eta$, and observe that there exist $B_1 \subseteq [n_1]$ and $B_2 \subseteq [n_2]$
such that $\eta <\mu_1(B_1) \mik 2\eta$ and $\eta < \mu_2(B_2) \mik 2\eta.$ Then,
\begin{eqnarray*}
\int_A f\,d\bmu & \mik & \int_{(A_1 \cup B_1 )\times (A_2 \cup B_2)} f \,d\bmu \\
& \stackrel{\eqref{eq3.3}}{\mik} & C\, \|f\|_{L_1} \big(\bmu\big( (A_1 \cup B_1) \times (A_2 \cup B_2)\big) + 2\eta\big)^{1/q} \\
& \mik & C\, \|f\|_{L_1} (\bmu(A) + 8\eta^2+2\eta)^{1/q}\mik C\, \|f\|_{L_1} (\bmu(A) + 6\eta)^{1/q}
\end{eqnarray*}
and the proof of the lemma is completed.
\end{proof}
Lemmas \ref{lem3.1} and \ref{lem3.2} will be used in the proof of the following result.
\begin{lem} \label{lem3.3}
There exist an algorithm and a polynomial\, $\Pi_2$ such that the following holds. Let $0 < \ee <1/2$ and $C\meg 1$.
Let $1 < p \mik \infty$, set $p^\dagger=\min \{2,p\}$ and let $q $ denote the conjugate exponent of $p^\dagger$. Also
let $a_0$ be as in Proposition \ref{Naor}, and set
\[ \vartheta= \frac{a_0 \, \ee}{16 C} \ \text{ and } \ \eta \mik\Big(\vartheta\cdot \iota(\calp)^{{\frac{2}{p^\dagger}}+1}\Big)^q.\]
If we input
\begin{itemize}
\item[$\mathtt{INP1}$:] a partition $\calp$ of\, $[n_1]\times [n_2]$ with $\calp \subseteq \cals$,
\item[$\mathtt{INP2}$:] a subset $A$ of\, $[n_1]\times[n_2]$ with $A \in \cals,$ and
\item[$\mathtt{INP3}$:] a $(C,\eta,p)$-regular matrix $f \colon [n_1]\times [n_2] \to \{0,1\},$
\end{itemize}
then the algorithm outputs
\begin{itemize}
\item[$\mathtt{OUT1}$:] a refinement $\calq$ of $\calp$ with $\calq\! \subseteq\! \cals$, $|\calq|\! \mik\! 4 |\calp|$ and
$\iota(\calq) \meg (\vartheta\cdot \iota(\calp)^{{\frac{2}{p^\dagger}}+1})^q,$ and
\item[$\mathtt{OUT2}$:] a set $B \in \cala_\calq$ such that
\begin{equation} \label{eq3.4}
\int_{A \triangle B} \! \ave(f \,| \, \cala_\calp) \,d\bmu \mik  2 C\, \|f\|_{L_1} \vartheta \  \text{ and } \
\int_{A \triangle B} \! f \,d\bmu \mik 6 C\, \|f\|_{L_1} \vartheta.
\end{equation}
\end{itemize}
If we additionally assume that the matrix $f$ in\, $\mathtt{INP3}$ satisfies
\begin{equation} \label{eq3.5}
\big| \int_A \!\big(f - \ave(f\,|\, \cala_\calp) \big)\, d\bmu\big| \meg a_0\, \ee\, \|f\|_{L_1},
\end{equation}
then the partition $\calq$ in\, $\mathtt{OUT2}$ satisfies
\begin{equation} \label{eq3.6}
\|\ave(f \, | \, \cala_\calq) -\ave(f \, | \, \cala_\calp)\|_{L_{p^\dagger}} \meg \frac{a_0\, \ee\, \|f\|_{L_1} }{2}.
\end{equation}
Finally, this algorithm has running time $|\calp| \cdot \Pi_2(n_1\cdot n_2).$
\end{lem}
Lemma \ref{lem3.3} is an algorithmic version of \cite[Lemmas 5.1 and 5.2]{DKK2}. We notice that if the matrix $f$ satisfies the estimate
in \eqref{eq3.5}, then inequality \eqref{eq3.6} implies that the partition $\calq$ is a genuine refinement of $\calp$. We also point out that
the polynomial $\Pi_2$ obtained by Lemma \ref{lem3.3} is absolute and independent of the parameters $\ee, C$ and~$p$. We proceed to the proof.
\begin{proof}[Proof of Lemma \ref{lem3.3}]
We may (and we will) assume that $A$ is nonempty. We select $A_1 \subseteq [n_1]$ and $A_2 \subseteq [n_2]$ such that $A =A_1 \times A_2,$ and we set
\[\theta= \vartheta^q \cdot \iota(\calp)^{\frac{2q}{p^\dagger}}.\]
Also let
\[\begin{split}
&\calp^1 = \{ P=P_1 \times P_2 \in \calp \colon \mu_1(A_1 \cap P_1) < \theta \mu_1(P_1) \ \mbox{and} \ \mu_2(A_2 \cap P_2) < \theta \mu_2(P_2) \},\\
&\calp^2 = \{ P=P_1 \times P_2 \in \calp \colon \mu_1(A_1 \cap P_1) < \theta \mu_1(P_1)\ \mbox{and} \ \mu_2(A_2 \cap P_2) \meg \theta \mu_2(P_2)\}, \\
&\calp^3 = \{ P=P_1 \times P_2 \in \calp \colon \mu_1(A_1 \cap P_1) \meg \theta \mu_1(P_1)\ \mbox{and} \ \mu_2(A_2 \cap P_2) < \theta \mu_2(P_2) \}, \\
&\calp^4 = \{ P=P_1 \times P_2 \in \calp \colon \mu_1(A_1 \cap P_1) \meg \theta \mu_1(P_1)\ \mbox{and} \ \mu_2(A_2 \cap P_2) \meg \theta \mu_2(P_2)\}.
\end{split}\]
Clearly, the family $\{\calp^1,\calp^2,\calp^3,\calp^4\}$ is a partition of $\calp$.

Now for every $P \in \calp$ we perform the following subroutine. First, assume that $P \in \calp^1 \cup \calp^2 \cup \calp^3$ and notice
that in this case we have $\bmu(A \cap P ) \mik \theta\bmu(P)$. Then we set $B_P = \emptyset$ and $\calq_P=\{P\}$. On the other hand,
if $P =P_1 \times P_2 \in \calp^4$, then we apply Lemma \ref{lem3.1} for $X_1 = P_1$ and $X_2 = P_2$, and we obtain\footnote{Notice that
for every $A\subseteq X_1$ we have $\mu_{X_1}(A)=\mu_1(A)/\mu_1(X_1)$, and similarly for $X_2$.} a partition $\calq_P$ of $P$ with
$\calq \in \cals$, $|\calq_P| \mik 4$ and $\iota(\calq_P) \meg \theta \cdot \iota(\calp)$, and a set $B_P \in \calq_P$ such that $A \cap P\subseteq B_P$
and $\bmu(B_P \setminus (A \cap P)) \mik  2 \theta \bmu(P).$

Once this is done, the algorithm outputs
\[\calq= \bigcup_{P \in \calp} \calq_P \ \text{ and } \ B=\bigcup_{P \in \calp} B_P.\]
Notice that there exists a polynomial $\Pi_2$ such that this algorithm has running time $|\calp| \cdot \Pi_2(n_1\cdot n_2)$. Indeed,
recall that the algorithm in Lemma \ref{lem3.1} runs in polynomial time and observe that we have applied Lemma \ref{lem3.1} at most $|\calp|$ times.

We proceed to show that the partition $\calq$ and the set $B$ satisfy the requirements of the lemma. To this end, we first observe that $\calq$ satisfies
the requirements in~$\mathtt{OUT1}$. Moreover, we have $B \in \cala_\calq$ and
\begin{equation} \label{eq3.7}
A\, \triangle\, B = \Big( \bigcup_{i=1}^3\bigcup_{P \in \calp^i} (A \cap P)\Big) \cup \Big(\bigcup_{P \in \calp^4} \big(B_P \setminus (A \cap P)\big)\Big).
\end{equation}
Therefore,
\begin{equation} \label{eq3.8}
\bmu(A\, \triangle\, B) \mik 2 \theta
\end{equation}
and so, by the $L_p$ regularity of $f$, H\"{o}lder's inequality, the monotonicity of the $L_p$~norms and the fact that $p^\dagger \mik p$, we obtain that
\[\begin{split}
\int_{A \triangle B} \!\! \ave(f \,| \, \cala_\calp) \,d\bmu & \mik
\|\ave(f\,|\, \cala_\calp)\|_{L_{p^\dagger}} \cdot \bmu(A\, \triangle\, B)^{1/q} \mik
\|\ave(f\,|\, \cala_\calp)\|_{L_{p}} \cdot \bmu(A\, \triangle\, B)^{1/q} \\
& \mik C\, \|f\|_{L_1} (2 \theta)^{1/q} \mik  2 C\, \|f\|_{L_1} \vartheta
\end{split}\]
which proves the first inequality in \eqref{eq3.4}. For the second inequality, by \eqref{eq3.7}, we have
\begin{equation} \label{eq3.9}
\int_{A \triangle B} \! f \,d\bmu =\sum_{P \in \calp^1 \cup \calp^2 \cup \calp^3} \int_{A \cap P } \! f \, d\bmu +
\sum_{P \in \calp^4} \int_{B_P \setminus (A \cap P)} \! f\, d\bmu
\end{equation}
and, by the definition of $\theta$ and the fact that $\eta \mik (\vartheta\cdot \iota(\calp)^{{\frac{2}{p^\dagger}}+1})^q$, we have
$\eta \mik \theta \bmu(P)$ for every $P \in \calp.$ Thus, if $P \in \calp^1 \cup \calp^2 \cup \calp^3$, then, by Lemma \ref{lem3.2} and
our assumption that $f$ is $(C,\eta,p)$-regular (and, consequently, $(C,\eta,p^\dagger)$-regular), we have
\[\int_{A \cap P} \! f\,d\bmu \mik C\, \|f\|_{L_1} ( \bmu(A \cap P) + 6 \eta)^{1/q} \mik 3 C\, \|f\|_{L_1} \big(\theta \bmu(P)\big)^{1/q}\]
which yields that
\begin{equation} \label{eq3.10}
\sum_{P \in \calp^1 \cup \calp^2 \cup \calp^3}
\int_{A \cap P} \! f\,d\bmu  \mik 3 C\, \|f\|_{L_1}\, \theta^{1/q} \sum_{P \in \calp^1 \cup \calp^2 \cup \calp^3} \bmu(P) ^{1/q}.
\end{equation}
On the other hand, by the choice of the family $\{B_P:P \in \calp^4\}$ and Lemma \ref{lem3.2},
\begin{equation} \label{eq3.11}
\sum_{P \in \calp^4} \int_{B_P \setminus (A \cap P)} \! f\, d\bmu \mik  6 C\, \|f\|_{L_1}\, \theta^{1/q} \sum_{P \in \calp^4} \bmu(P)^{1/q}.
\end{equation}
Moreover, since $q \meg 2$ we have that $x ^{1/q}$ is concave on $\rr^+$, and so
\begin{equation} \label{eq3.12}
\sum_{P \in \calp} \bmu(P)^{1/q} \mik |\calp|^{\frac{1}{p^\dagger}} \mik \iota (\calp)^{-\frac{2}{p^\dagger}}.
\end{equation}
Combining \eqref{eq3.10}--\eqref{eq3.12}, we see that the second inequality in \eqref{eq3.4} is satisfied.

Finally, assume that the matrix $f$ satisfies \eqref{eq3.5}. By \eqref{eq3.4} and the choice of $\vartheta$,
\[\begin{split}
\Big| \int_A \big(f-\ave(f \,|\, \cala_\calp)&\big) \, d\bmu - \int_B \big(f-\ave(f \,|\, \cala_\calp)\big) \, d\bmu\Big|\\
&\mik \int_{A \triangle B} \! \ave(f\, | \,\cala_\calp) \, d\bmu
+\int_{A \triangle B} \! f \,d\bmu \mik \frac{a_0\, \ee\, \|f\|_{L_1} }{2} \end{split}\]
and so, by \eqref{eq3.5}, we have
\begin{equation} \label{eq3.13}
\Big| \int_B \big(f-\ave(f \,|\, \cala_\calp)\big)\, d\bmu \Big| \meg \frac{a_0\, \ee\, \|f\|_{L_1} }{2}.
\end{equation}
Moreover, the fact that $B \in \cala_\calq$ yields that
\begin{equation} \label{eq3.14}
\int_B \big( f - \ave(f \, | \, \cala_\calp)\big) \, d\bmu=
\int_B \big(\ave(f \, | \, \cala_\calq) - \ave(f \, | \, \cala_\calp)\big) \, d\bmu.
\end{equation}
Thus, by the monotonicity of the $L_p$ norms, we conclude that
\[\begin{split}
& \|\ave(f \, | \, \cala_\calq) - \ave(f \, | \, \cala_\calp)\|_{L_{p^\dagger}} \meg
\|\ave(f \, | \, \cala_\calq) - \ave(f \, | \, \cala_\calp)\|_{L_1} \\
& \meg \Big| \int_B \big( \ave(f \, | \, \cala_\calq) - \ave(f \, | \, \cala_\calp)\big) \, d\bmu\Big| \stackrel{\eqref{eq3.14}}{=}
\Big|\int_B \big(f - \ave(f \, | \, \cala_\calp)\big) \, d\bmu\Big| \stackrel{\eqref{eq3.13}}{\meg} \frac{a_0\, \ee\, \|f\|_{L_1}}{2}
\end{split}\]
and the proof of Lemma \ref{lem3.3} is completed.
\end{proof}


\section{Proof of theorem \ref{thm1.3}}

\numberwithin{equation}{section}

We will describe a recursive algorithm that performs the following steps. Starting from the trivial partition of $[n_1]\times [n_2]$ and using
Lemma \ref{lem3.3} as a subroutine, the algorithm will produce an increasing family of partitions of $[n_1]\times [n_2]$. Simultaneously, using
Proposition \ref{Naor} as a subroutine, the algorithm will be checking if the partition that is produced at each step satisfies the requirements
in $\mathtt{OUT}$ of Theorem~\ref{thm1.3}. The fact that this algorithm will eventually terminate is based on Proposition~\ref{prop2.1}.
\begin{proof}[Proof of Theorem \ref{thm1.3}]
Let $a_0$ be as in Proposition \ref{Naor}, and set
\begin{equation} \label{eq4.1}
\vartheta= \frac{a_0\, \ee}{16C}, \ \
\tau=\Big\lceil \frac{4 C^2}{(p^\dagger-1)\, \ee^2\, a_0^2}\Big\rceil \ \text{ and } \
\eta=\vartheta^{\sum_{i=1}^{\tau+1} (\frac{2}{p^\dagger}+1)^{i-1}q^i}.
\end{equation}
Also fix a $(C,\eta,p)$-regular matrix $f \colon [n_1] \times [n_2] \to \{0,1\}$. The algorithm performs the following steps.
\medskip

\noindent $\mathtt{InitialStep}$: We set $\calp_0 \coloneqq\{[n_1] \times [n_2]\}$ and we apply the algorithm in Proposition~\ref{Naor}
for the matrix $f-\ave(f \,|\,\cala_{\calp_0})$. Thus, we obtain a set $A_0 \subseteq [n_1]\times [n_2]$ with $A_0 \in \cals$ and such that
$( n_1 \cdot  n_2 ) |\int_{A_{0} } \big(f - \ave(f \,|\,\cala_{\calp_0})\big) \,d\bmu|  \meg a_0 \|f - \ave(f \,|\,\cala_{\calp_0})\|_{\square}.$
If $|\int_{A_{0}} \big(f - \ave(f \,|\,\cala_{\calp_0})\big) \,d\bmu| \mik a_0\, \ee \, \|f\|_{L_1},$ then the algorithm outputs the partition
$\calp_0$ and $\mathtt{Halts}$. Otherwise, the algorithm sets $m=1$ and enters into the following loop.
\medskip

\noindent $\mathtt{GeneralStep}$: The algorithm will have as an input a positive integer $m \in [\tau-1]$, a partition\footnote{Notice that
$\calp_0\subseteq \cals$ and $\iota(\calp_0) =1$.} $\calp_{m-1}\subseteq \cals$ and a set $A_{m-1} \subseteq [n_1]\times [n_2]$ with
$A_{m-1} \in \cals$, such that
\begin{itemize}
\item[(a)] $|\calp_{m-1}| \mik 4^{m}$,
\item[(b)] $(\vartheta \cdot\iota(\calp_{m-1})^{\frac{2}{p^\dagger}+1} )^q \meg \vartheta^{\sum_{i=1}^{m} (\frac{2}{p^\dagger} +1 )^{i-1} q^i}$, and
\item[(c)] $|\int_{A_{m-1}} \big(f- \ave(f\,|\,\cala_{\calp_{m-1}})\big)\, d\bmu|> a_0\, \ee \, \|f\|_{L_1}.$
\end{itemize}
By (b) and the choice of $\eta$ in \eqref{eq4.1}, we have $\eta \mik (\vartheta \cdot\iota(\calp_{m-1})^{\frac{2}{p^\dagger}+1})^q$.
This fact together with the choice of $\vartheta$ in \eqref{eq4.1} allows us to perform the algorithm in Lemma~\ref{lem3.3} for the matrix $f$,
the partition $\calp_{m-1}$ and the set $A_{m-1}$. Thus, we obtain a refinement $\calp_m$ of $\calp_{m-1}$ with $\calp_m \subseteq \cals$,
$|\calp_m| \mik 4 |\calp_{m-1}|,$ $\iota(\calp_m)\meg (\vartheta \cdot \iota(\calp_{m-1})^{\frac{2}{p^\dagger}+1})^q$, such that
\[\|\ave(f\,|\,\cala_{\calp_{m}}) -\ave(f\,|\,\cala_{\calp_{m-1}}) \|_{L_{p^\dagger}} \meg \frac{a_0\,\ee\, \|f\|_{L_1}}{2}.\]
Next, we apply the algorithm in Proposition \ref{Naor} for the matrix $f - \ave(f \,|\,\cala_{\calp_{m}})$, and we obtain a set
$A_m \subseteq [n_1]\times [n_2]$ with $A_m \in \cals$ and such that
\[(n_1 \cdot n_2)\, \Big| \int_{A_{m} } \big(f - \ave(f \,|\,\cala_{\calp_{m}})\big) \,d\bmu\Big|
\meg a_0\, \|f - \ave(f \,|\,\cala_{\calp_{m}})\|_{\square}.\]
If $|\int_{A_{m}} \big(f - \ave(f \,|\,\cala_{\calp_{m}})\big) \,d\bmu|\mik a_0\, \ee\, \|f\|_{L_1}$, then the algorithm outputs the partition
$\calp_{m}$ and $\mathtt{Halts}$. Otherwise, if $m < \tau-1$, then the algorithm reruns the loop we described above for the positive integer $m+1$,
the partition $\calp_m$ and the set $A_m$, while if $m=\tau-1$, then the algorithm proceeds to the following step.
\medskip

\noindent $\mathtt{FinalStep}$: The algorithm will have as an input a partition $\calp_{\tau-1} \subseteq \cals$ and a set
$A_{\tau-1} \subseteq [n_1]\times [n_2]$ with $A_{\tau-1}\in \cals$, such that
\begin{itemize}
\item[(d)] $|\calp_{\tau-1}| \mik 4^{\tau-1}$,
\item[(e)] $(\vartheta \cdot \iota(\calp_{\tau-1}) ^{\frac{2}{p^\dagger}+1})^q \meg \vartheta^{\sum_{i=1}^\tau (\frac{2}{p^\dagger}+1)^{i-1} q^i},$ and
\item[(f)] $|\int_{A_{\tau-1}} \big(f - \ave(f \,|\,\cala_{\calp_{\tau-1}})\big) \, d\bmu|> a_0 \, \ee\, \|f\|_{L_1}.$
\end{itemize}
Again observe that, by (e) and the choice of $\eta$ in \eqref{eq4.1}, we have
$\eta \mik(\vartheta \cdot \iota(\calp_{\tau-1}) ^{\frac{2}{p^\dagger}+1})^q$. Using this fact and the choice of $\vartheta$ in \eqref{eq4.1},
we may apply the algorithm in Lemma~\ref{lem3.3} for the matrix $f,$ the partition $\calp_{\tau-1}$  and the set $A_{\tau-1}$. Therefore, we obtain
a refinement $\calp_\tau$ of $\calp_{\tau-1}$ with $\calp_\tau \subseteq \cals$, $|\calp_\tau| \mik 4 |\calp_{\tau-1}|$,
$\iota(\calp_\tau)\meg (\vartheta \cdot \iota(\calp_{\tau-1})^{\frac{2}{p^\dagger}+1})^q,$ and such that
\[\|\ave(f\,|\,\cala_{\calp_{\tau}}) -\ave(f\,|\,\cala_{\calp_{\tau-1}}) \|_{L_{p^\dagger}} \meg\frac{a_0\,\ee\, \|f\|_{L_1}}{2}.\]
The algorithm outputs the partition $\calp_\tau$ and $\mathtt{Halts}$.
\medskip

Notice that there exists a polynomial $\Pi_0$ such that the previous algorithm has running time $(\tau 4^\tau)\cdot \Pi_0(n_1\cdot n_2)$.
Indeed, by Proposition \ref{Naor}, there exists a polynomial $\Pi'_0$ such that the $\mathtt{Initial Step}$ runs in time $\Pi'_0(n_1\cdot n_2).$
Moreover, by the running times of the algorithms in Lemma  \ref{lem3.3} and Proposition \ref{Naor}, there exists a polynomial $\Pi''_0$ such that
each of the $\mathtt{GeneralStep}$ runs in time $4^\tau\cdot \Pi''_0(n_1\cdot n_2)$. Finally, invoking again Lemma \ref{lem3.3}, we see that there
exists a polynomial $\Pi'''_0$ such that the $\mathtt{FinalStep}$ runs in time $\Pi'''_0(n_1\cdot n_2).$ Therefore, the algorithm we described above
runs in time
\[\Pi'_0(n_1\cdot n_2) + (\tau -1)\, 4^\tau\, \Pi''_0(n_1\cdot n_2) + \Pi'''_0(n_1\cdot n_2) \]
which in turn yields that there exists a polynomial $\Pi_0$ such that the algorithm has running time $(\tau \, 4^\tau) \cdot\Pi_0(n_1\cdot n_2).$
\medskip

It remains to verify that the previous algorithm will produce a partition that satisfies the requirements in $\mathtt{OUT}$ of Theorem \ref{thm1.3}.
As we have noted, the argument is based on Proposition \ref{prop2.1} and can be seen as the $L_p$ version of the, so called, \textit{energy increment method}
(see, e.g., \cite[Lemmas 10.40 and 11.31]{TV}). For more information and further applications of this method we refer to \cite{DKK1,DKK2,DKT}.

We proceed to the details. First assume that the algorithm has stopped before the $\mathtt{FinalStep}$. Then the output of the algorithm is one of
the partitions we described in $\mathtt{InitialStep}$ and in $\mathtt{GeneralStep}$, say $\calp_m$ for some $m\in\{0,\dots,\tau-1\}$. Observe that $\calp_m$
satisfies $\calp_m \subseteq \cals$, $|\calp_m|\mik 4^m,$ and $\iota(\calp_m) \meg \eta$; in other words, $\calp_m$ satisfies the first three requirements
in $\mathtt{OUT}$ of Theorem \ref{thm1.3}. Moreover, recall that there exists a set $A_m \subseteq [n_1]\times [n_2]$ with $A_m \in \cals,$ and such that
\[(n_1 \cdot n_2) \Big| \int_{A_m} \big(f - \ave(f\,|\,\cala_{\calp_m})\big)\,d\bmu\Big| \meg a_0\, \|f - \ave(f\,|\,\cala_{\calp_m})\|_{\square}. \]
On the other hand, since the output of the algorithm is the partition $\calp_m$, we have
$|\int_{A_m} \big(f - \ave(f\,|\,\cala_{\calp_m})\big)\,d\bmu| \mik a_0 \,\ee \, \|f\|_{L_1}$. Combining these estimates, we conclude that
$\|f - \ave(f\,|\,\cala_{\calp_m})\|_{\square} \mik \ee \|f\|_{\square}$.

Next, assume that the algorithm reaches the $\mathtt{FinalStep}$. Recall that $\calp_{\tau}\subseteq\cals$ and observe that, by (d) above and
the fact that $|\calp_{\tau}| \mik 4|\calp_{\tau-1}|$, we have $|\calp_\tau| \mik 4^\tau$. Moreover, by (e) and the choice of $\eta$ in \eqref{eq4.1},
\begin{equation} \label{eq4.2}
\iota(\calp_\tau)\meg (\vartheta \cdot \iota(\calp_{\tau-1})^{\frac{2}{p^\dagger}+1})^q\meg
\vartheta^{\sum_{i=1}^{\tau} (\frac{2}{p^\dagger}+1)^{i-1}q^i} \meg\eta.
\end{equation}
Thus, we only need to show that  $\|f - \ave(f \, | \, \cala_{\calp_\tau})\|_{\square} \mik \ee \|f\|_{\square}$. To this end assume, towards
a contradiction, that $\|f - \ave(f \, | \, \cala_{\calp_\tau})\|_{\square} > \ee \|f\|_{\square}$. Notice that, by the choice of $\eta$ in
\eqref{eq4.1} and \eqref{eq4.2}, we have $( \vartheta \cdot \iota(\calp_\tau)^{\frac{2}{p^\dagger}+1})^q \meg \eta$. Using the previous two estimates,
Proposition \ref{Naor}, Lemma \ref{lem3.3} and arguing precisely as in the $\mathtt{GeneralStep}$, we may select a refinement $\calp_{\tau+1}$ of\, $\calp_\tau$
with $\calp_{\tau+1}\subseteq \cals$ and $\iota(\calp_{\tau+1}) \meg \eta$, and such that
$\|\ave(f \,|\,\cala_{\calp_{\tau +1}}) - \ave(f \, | \, \cala_{\calp_{\tau}})\|_{L_{p^{\dagger}}} \meg (a_0\, \ee\, \|f\|_{L_1})/2$.
It follows that there exists an increasing finite sequence $(\calp_i)_{i=0}^{\tau +1}$ of partitions with $\calp_0=\{[n_1]\times[n_2]\}$
and such that for every $i\in [\tau+1]$ we have $\calp_i\subseteq\cals$, $\iota(\calp_i) \meg \eta$, and
\begin{equation} \label{eq4.3}
\|\ave(f\,|\,\cala_{\calp_{i}}) - \ave(f\,|\,\cala_{\calp_{i-1}})\|_{L_{p^\dagger}} \meg \frac{a_0\, \ee\, \|f\|_{L_1} }{ 2}.
\end{equation}
Now set $d_0= \ave(f\,|\,\cala_{\calp_0})$ and $d_i = \ave(f\,|\,\cala_{\calp_{i}}) - \ave(f\,|\,\cala_{\calp_{i-1}})$ for every $i\in [\tau+1]$,
and observe that the sequence $(d_i)_{i=0}^{\tau+1}$ is a martingale difference sequence. Therefore, by Proposition \ref{prop2.1} and the fact that
the matrix $f$ is $(C,\eta,p)$-regular, we have
\begin{eqnarray*}
\frac{a_0 \, \ee \, \|f\|_{L_1}}{2} \cdot \sqrt{\tau+1} & \stackrel{\eqref{eq4.3}}{\mik} &
\Big(\sum_{i=1}^{\tau+1} \|d_i\|_{L_{p^\dagger}}^2\Big)^{1/2} \mik \Big(\sum_{i=0}^{\tau+1} \|d_i\|_{L_{p^\dagger}}^2\Big)^{1/2} \\
& \stackrel{\eqref{eq2.1}}{\mik} & \frac{1}{\sqrt{p^\dagger-1}}\, \big\| \sum_{i=0}^{\tau+1} d_i\big\|_{L_{p^\dagger}} =
\frac{1}{\sqrt{p^\dagger-1}}\, \|\ave(f\,|\,\cala_{\calp_{\tau+1}})\|_{L_{p^\dagger}} \\
& \mik & \frac{C}{\sqrt{p^\dagger-1}}\, \|f\|_{L_1}
\end{eqnarray*}
which clearly contradicts the choice of $\tau$ in \eqref{eq4.1}. The proof of Theorem \ref{thm1.3} is thus completed.
\end{proof}


\section{Applications}

\numberwithin{equation}{section}

\subsection{Tensor approximation algorithms}

Throughout this subsection let $k\meg 2$ be an integer. Also let $n_1,\dots,n_k$ be positive integers, and let $\bmu_k$
denote the uniform probability measure on $[n_1]\times \dots \times [n_k]$.

Recall that a $k$-\emph{dimensional tensor} is a function $F \colon [n_1] \times \dots \times [n_k]\to \rr$. (Notice, in particular,
that a $2$-dimensional tensor is just a matrix.) Also recall, that a tensor $G \colon [n_1] \times \dots \times [n_k]\to\rr$ is called a
\emph{cut tensor} if there exist a real number $c$ and for every $i\in [k]$ a subset $S_i$ of $[n_i]$ such that
$G=c \cdot \boldsymbol{1}_{S_1\times\cdots\times S_k}$. Finally, recall that for every tensor $F \colon [n_1]\times \dots \times [n_k] \to \rr$
its \emph{cut norm} is defined as
\[\|F\|_{\square} = \Big(\prod_{i=1}^k n_i\Big) \cdot \max\Big\{ \Big|\int_{S_1 \times \dots \times S_k} \! F \, d\bmu_k\Big|:
S_i \subseteq [n_i] \text{ for every } i\in [k] \Big\}.\]
Next, set
\begin{equation} \label{eq5.1}
k_1 \coloneqq \lfloor k/2 \rfloor, \ \ A_k\coloneqq [n_1]\times \dots \times [n_{k_1}] \ \text{ and } \
B_k\coloneqq [n_{k_1 + 1}]\times \dots \times [n_k], \end{equation}
and for every tensor $F \colon [n_1] \times \dots \times [n_k] \to \{0,1\}$ let the \emph{respective matrix} $f_F$ of $F$ be the matrix
$f_{F}\colon A_k\times B_k\to\{0,1\}$ defined by the rule
\begin{equation} \label{eq5.2}
f_{F}\big((i_1,\dots,i_{k_1}),(i_{k_1 + 1} ,\dots,i_k)\big) = F(i_1,\dots,i_k)
\end{equation}
for every $\big((i_1,\dots,i_{k_1}),(i_{k_1 + 1},\dots,i_k)\big)\in A_k\times B_k= [n_1]\times \dots \times [n_k]$.

As in \cite{CCF}, we extend the notion of $L_p$ regularity from matrices to tensors as follows.
\begin{defn}[$L_p$ regular tensors] \label{defn5.1}
Let $0<\eta \mik 1, C\meg 1$ and $1 \mik p \mik \infty.$ A~tensor $F \colon [n_1]\times \dots \times [n_k]$ is called \emph{$(C,\eta,p)$-regular}
if its respective matrix $f_{F}$ is $(C,\eta,p)$-regular, that is, if for every partition $\calp$ of $A_k\times B_k$ with
$\calp\subseteq\cals_{A_k\times B_k}$ and $\iota(\calp)\meg \eta$ we have $\|\ave(f_{F}\, |\, \cala_{\calp})\|_{L_p}\mik C$.
\end{defn}
To state our main result about $L_p$ regular tensors we need to introduce some numerical invariants. Specifically, let $\ee >0$ and $C\meg 1$.
Also let $1 < p \mik \infty,$ set $p^\dagger=\min \{2,p\}$ and let $q$ denote the conjugate exponent of $p^\dagger$. Finally, let $a_1,a_2$
be as in Theorem \ref{thm1.3}, and define
\begin{equation} \label{eq5.3}
\tau(\ee,C,p)=\Big\lceil\frac{a_1 \, C^2}{(p^\dagger-1) \,\ee^2}\Big\rceil \ \text{ and } \
\eta(\ee,C,p) = \Big(\frac{a_2 \,\ee}{C} \Big)^{\sum_{i=1}^{\tau(\ee,C,p)+1} (\frac{2}{p^\dagger}+1)^{i-1}q^i}.
\end{equation}
We have the following theorem.
\begin{thm} \label{thm5.2}
There exist a constant $b$, an algorithm and a polynomial\, $\Pi_3$ such that the following holds. Let $0< \ee < 1/2$ and $C\meg 1$.
Also let $1 < p \mik \infty$, and let $\tau=\tau(\ee/2,C,p)$ and $\eta=\eta(\ee/2,C,p)$ be as in \eqref{eq5.3}. If we input
\begin{itemize}
\item[$\mathtt{INP}$:] a $(C,\eta,p)$-regular tensor $F \colon [n_1] \times \dots \times [n_k] \to \{0,1\}$,
\end{itemize}
then the algorithm outputs
\begin{itemize}
\item[$\mathtt{OUT}$:] cut tensors $G_1,\dots,G_s$ with $\displaystyle{s\mik \Big(\frac{ 2b\, C}{\ee\, \eta^2}\Big)^{2(k-1)}}$ and such that
\begin{equation} \label{eq5.4}
\big\|F - \sum_{i=1}^s G_i\big\|_{\square} \mik \ee \|F\|_{\square} \ \text{ and }  \
\sum_{i=1}^s \|G_i\|_{L_\infty}^2 \mik \Big(\frac{C\, \|F\|_{L_1}}{\eta^2}\Big)^2 \, b^{2k}.
\end{equation}
\end{itemize}
This algorithm has running time $\big(\tau\, 4^\tau + \big(\frac{2C}{\ee \eta^2}\big)^{3k}\big) \cdot \Pi_3\big(\prod_{i=1}^k n_i\big)$.
\end{thm}
Theorem \ref{thm5.2} can be proved arguing precisely as in the proof of \cite[Theorem 2]{CCF} and using Theorem \ref{thm1.3} instead
of \cite[Corollary 1]{CCF}. We leave the details to the interested reader.

\subsection{MAX-CSP instances approximation}

In what follows let $n,k$ denote two positive integers with $k \mik n$.

Let $V=\{x_1,\dots,x_n\}$ be a set of Boolean variables, and recall that an \emph{assignment} $\sigma$ on $V$ is a map $\sigma \colon V \to \{0,1\}$.
Notice that if $\sigma$ is an assignment on $V$ and $W \subseteq V$, then $\sigma|_W \colon W \to \{0,1\}$ is an assignment on $W$. Also recall that
a $k$-\emph{constraint} is a pair $(\phi,V_\phi)$ where $V_\phi \subseteq V$ with $|V_\phi|=k$ and $\phi \colon \{0,1\}^{V_\phi} \to \{0,1\}$ is a not
identically zero map. Finally, recall that  a $k$-\emph{CSP instance} over $V$ is a family $\calf$ of $k$-constraints over $V$.

For every $k$-CSP instance $\calf$ we define
\begin{equation} \label{eq5.5}
\mathrm{OPT}(\calf) = \max_{\sigma \in \{0,1\}^V} \sum_{(\phi,V_\phi) \in \calf} \phi(\sigma|_{V_\phi}).
\end{equation}
Moreover, let $\Psi_k$ be the set of all non-zero maps from $\{0,1\}^k$ into $\{0,1\}$. We have the following definition.
\begin{defn} \label{defn5.3}
Let $\psi \in \Psi_k$. Also let $(\phi,V_\phi)$ be a $k$-constraint over $V$ where $V_{\phi}=\{x_{i_1},\dots,x_{i_k}\}$ for some
$1 \mik i_1 < \cdots < i_k \mik n$. We say that $(\phi,V_\phi)$ is of \emph{type $\psi$} if for every assignment $\sigma \colon V \to \{0,1\}$ we have
\[\psi\big(\sigma(x_{i_1}),\dots,\sigma(x_{i_k})\big) = \phi(\sigma|_{V_\phi}).\]
\end{defn}
Observe that every $k$-CSP instance $\calf$ can be represented by a family $(F_{\calf}^\psi)_{\psi \in \Psi_k}$
of $2^{2^k}-1$ tensors where for every $\psi \in \Psi_k$ the tensor $F_\calf^\psi \colon [n]^k \to \{0,1\}$ is defined by the rule
\begin{equation} \label{eq5.6}
F_\calf^\psi(i_1,\dots,i_k) =
\begin{cases}
1 & \text{if there is } (\phi,V_\phi) \in \calf \text{ of type } \psi \\
  & \text{with } V_\phi = \{x_{i_1},\dots,x_{i_k}\}, \\
0 & \text{otherwise}.
\end{cases}
\end{equation}
Having this representation in mind, we say that a $k$-constraint $\calf$ is $(C,\eta,p)$-regular for some $0< \eta\mik 1$, $C\meg 1$ and
$1 \mik p \mik \infty$, provided that for every $\psi \in \Psi_k$ the tensor $F_\calf^\psi$ defined above is $(C,\eta,p)$-regular.

We have the following theorem which extends \cite[Theorem 3]{CCF}. It follows from Theorem \ref{thm5.2} using the arguments
in the proof of \cite[Theorem 3]{CCF}; as such, its proof is left to the reader.
\begin{thm} \label{thm5.4}
There exist an algorithm, a constant $\gamma >0$ and a polynomial\, $\Pi_4$ such that the following holds. Let $k$ be a positive integer,
and let $0<\ee < 1/2$, $C\meg 1$ and $1<p\mik \infty$. Set $a=\ee\, 2^{-(2^k+2k+2)}$, and let $\tau=\tau(a,C,p)$ and $\eta=\eta(a,C,p)$
be as in \eqref{eq5.3}. If we input
\begin{itemize}
\item[$\mathtt{INP}$:] a $(C,\eta,p)$-regular $k$-CSP instance $\calf$ over a set $V=\{x_1,\dots,x_n\}$ of Boolean variables,
\end{itemize}
then the algorithm outputs
\begin{itemize}
\item[$\mathtt{OUT}$:] an assignment $\sigma \colon V \to \{0,1\}$ such that
\[ \sum_{(\phi,V_\phi) \in \calf} \phi(\sigma|_{V_\phi}) \meg (1 - \ee)\cdot \mathrm{OPT}(\calf). \]
\end{itemize}
This algorithm has running time
\[ \Pi_4\Bigg( n^k\cdot \exp \Big(k\, 2^k\, 2^{2^k} \big(\frac{2C}{\ee\,\eta^2}\big)^{2k} \ln\!\big( \frac{2C}{\ee\,\eta^2}\big)\Big) \Bigg).\]
\end{thm}


\end{document}